\documentclass[11pt]{article}
\usepackage{epigamath}

\usepackage[notext]{kpfonts}
\usepackage{baskervald}

\setpapertype{A4}

\usepackage[english]{babel}

\usepackage[pdftex]{graphicx}  

\title{Algebraic models of the Euclidean plane}
\titlemark{Algebraic models of the Euclidean plane}
\author{\vspace{0cm} J\'er\'emy Blanc and Adrien Dubouloz}
\authoraddresses{
\authordata{J\'er\'emy Blanc}{\firstname{J\'er\'emy} \lastname{Blanc}\\
\institution{Departement Mathematik und Informatik, Unversit\"at
Basel, Spiegelgasse 1, CH-4051
Basel, Switzerland}\\
\email{jeremy.blanc@unibas.ch}}\\
\authordata{Adrien Dubouloz}{\firstname{Adrien}
\lastname{Dubouloz}\\
\institution{IMB UMR5584, CNRS, Univ. Bourgogne Franche-Comt\'e,
F-21000 Dijon,
France}\\
\email{adrien.dubouloz@u-bourgogne.fr}} }
\authormark{J. Blanc and A. Dubouloz}
\date{\vspace{-5ex}} 
\journal{\'Epijournal de G\'eom\'etrie Alg\'ebrique} 
\acceptation{Received by the Editors on May 16, 2018, and in final
form on October 10, 2018. \\ Accepted on November 13, 2018.}

\newcommand{\articlereference}{Volume 2 (2018), Article Nr. 14}

\acknowledgement{The first author acknowledges support by the Swiss
National Science Foundation Grant \textquotedbl{}Birational
Geometry\textquotedbl{} PP00P2\_153026.}

\usepackage[all]{xy}

\allowdisplaybreaks

\numberwithin{equation}{numsection}

\newtheorem{theorem}{Theorem}
\newtheorem{lemma}{Lemma}[section]
\newtheorem{corollary}[lemma]{Corollary}
\newtheorem{proposition}[lemma]{Proposition}
\newtheorem{question}[lemma]{Question}
\newtheorem{definition}[lemma]{Definition}
\newtheorem{notation}[lemma]{Notation}
\newtheorem{remark}[lemma]{Remark}
\newtheorem{example}[lemma]{Example}

\usepackage{xypic,color,xy}
\xyoption{arc}
\usepackage{graphicx}
\newcommand{\Aut}{\mathrm{Aut}}
\newcommand{\iso }{\stackrel{\simeq}{\longrightarrow}}
\renewcommand{\k}{\mathbf{k}}
\renewcommand{\r}{\mathbb{R}}
\newcommand{\C}{\mathbb{C}}
\newcommand{\p}{\mathbb{P}}
\newcommand{\A}{\mathbb{A}}

\begin{document}


\maketitle



\begin{prelims}


\def\abstractname{Abstract}
\abstract{We introduce a new invariant, the real
$($logarithmic$)$-Kodaira dimension, that allows to distinguish
smooth real algebraic surfaces up to birational diffeomorphism. As
an application, we construct infinite families of smooth rational
real algebraic surfaces with trivial homology groups, whose real
loci are diffeomorphic to $\mathbb{R}^2$, but which are pairwise not
birationally diffeomorphic. There are thus infinitely many
non-trivial models of the euclidean plane, contrary to the compact
case.}



\vspace{0.15cm}

\languagesection{Fran\c{c}ais}{%

\textbf{Titre. Mod\`eles alg\'ebriques du plan euclidien}
\commentskip \textbf{R\'esum\'e.} Nous introduisons un nouvel
invariant, la dimension de Kodaira (\emph{logarithmique}) r\'eelle,
qui permet de distinguer les surfaces alg\'ebriques r\'eelles lisses
\`a diff\'eomorphismes birationnels pr\`es. En guise d'application,
nous construisons des familles infinies de surfaces alg\'ebriques
r\'eelles rationnelles lisses ayant des groupes d'homologie
triviaux, dont les lieux r\'eels sont diff\'eomorphes \`a
$\mathbb{R}^2$ mais qui sont deux \`a deux non birationnellement
diff\'eomorphes. Contrairement au cas compact, il y a donc une
infinit\'e de mod\`eles non triviaux du plan euclidien.}

\end{prelims}


\newpage

\setcounter{tocdepth}{1} \tableofcontents

\section*{Introduction}

A real quasi-projective algebraic variety $X$ can be viewed as a
complex quasi-projective algebraic variety endowed with an
anti-regular involution, or equivalently as a locally closed
subscheme of $\mathbb{P}^n_{\mathbb{C}}$ which is defined over $\r$.
We can then speak about the set $X(\r)$  of real points of $X$ (real
locus). If $X$ is smooth, this set is naturally endowed with the
structure of differential real manifold, and $X$ is said to be an
\emph{algebraic model} of this differential manifold. Two models
$X_1$ and $X_2$ of the same manifold are said to be equivalent if
there exists a diffeomorphism $X_1(\r)\to X_2(\r)$ which comes from
a birational map $\varphi\colon X_1\dasharrow X_2$, such that
$\varphi$ and $\varphi^{-1}$ are defined at each point of $X_1(\r)$
and $X_2(\r)$ respectively. Such a map is called a \emph{birational
diffeomorphism}. In general a manifold can admit plenty of different
models. For example, the hypersurfaces of $\mathbb{P}_{\mathbb{R}}^3$ given by
the equations $x^{2n}+y^{2n}+z^{2n}-t^{2n}=0$, $n\ge 1$,  provide
infinitely many models of the sphere $\mathbb{S}^2$ which are
pairwise not birational. Nevertheless, if one restricts to the
simplest ones, namely the rational models, then for smooth compact
manifolds of dimension at most $2$, the model is then unique. In
dimension~$1$, we obtain only $\mathbb{P}^1$ and in dimension $2$
this is the following result of Biswas and Huisman:

\bigskip

\noindent{\bf Theorem.} {\rm \cite[Corollary 8.1]{BH}} {\it A
compact connected real manifold of dimension $2$ admits a rational
model if and only if it is non-orientable or diffeomorphic to
$\mathbb{S}^{2}$ or $\mathbb{S}^{1}\times \mathbb{S}^{1}$. Moreover,
this model is unique, up to birational diffeomorphism.}

\bigskip

In the non-compact case, the real locus of the real affine algebraic
variety $\A^2_\r$ provides an obvious rational algebraic model of
the Euclidean plane $\r^2$ endowed with its standard structure of
differential manifold. It is easy to find plenty of other rational
models of $\r^2$: we can choose for instance the complement in
$\p^2_\r$ of a smooth irreducible real curve $\Gamma\subseteq \p^2_\r$ of
odd degree $d\ge 3$ such that $\Gamma(\r)$ is an oval equivalent to
a line by a diffeomorphism of $\r\p^2$. It is thus natural to
restrict the study of such models to the smaller class of ``Fake
real planes'', introduced in \cite{DM15} as being smooth algebraic
surfaces $S$ defined over $\r$, non isomorphic to $\A_{\r}^{2}$ but
whose real locus is diffeomorphic to $\r^{2}$ and whose
complexifications $S_{\C}$ have ``minimal topology'' in the sense
that they are \emph{$\mathbb{Q}$-acyclic} topological manifolds,
that is, topological manifolds whose singular homology groups with
rational coefficients $\tilde{H}_i(S_{\C};\mathbb{Q})$ are all
trivial.

By general results \cite{Fu82,GuP97,GuPS97,DM15} all these surfaces
are affine and rational. A partial classification of them as real
algebraic varieties was given in \cite{DM15}, according to their
usual Kodaira dimension. Families of fake real planes of each
Kodaira dimension $\kappa\in\{-\infty,0,1,2\}$ birationally
diffeomorphic to $\mathbb{A}^2_{\r}$ were constructed in
\cite{DMSpitz}. The existence of fake real planes non birationally
diffeomorphic to $\A_{\r}^{2}$ was left open.

Here we show that $\r^{2}$ admits algebraic models non birationally
diffeomorphic to $\A_{\r}^{2}$ of every Kodaira dimension
$\kappa=0,1,2$, answering the main question of \cite{DM15} :

\begin{theorem}\label{TheoremFirst}
There are infinitely many rational models $S$ of the plane $\r^{2}$
up to birational diffeomorphism, all having trivial reduced homology
groups $\tilde{H}_{i}(S_{\C};\mathbb{Q})$. Such models exist for
every $\kappa=0,1,2$, and moreover, for $\kappa=1,2$, there exist
infinitely many models $S$ up to birational diffeomorphism for which
$S_{\C}$ is even topologically contractible.
\end{theorem}

In order to prove this result, we define a notion of \emph{real
Kodaira dimension} $\kappa_\r(S)$ (Definition~\ref{Defi:KappaR}),
which has the property to be smaller than or equal to the classical
one $\kappa(S)$, and can be computed in a very similar way (see
Definition~\ref{Defi:KappaR} and Remark~\ref{Rem:Imagloops}).
Moreover, we have equality $\kappa(S)=\kappa_\r(S)$ in the natural
case where $S$ admits a smooth projective completion $V$ with SNC
boundary $B=V\setminus S$ consisting only of real curves isomorphic
to $\p^1_{\r}$, and intersecting only at real points. The main new
noteworthy feature of $\kappa_\r(S)$ is that it is invariant under
birational diffeomorphisms (Corollary~\ref{RealKodIsIvariant}),
contrary to $\kappa(S)$ (Example~\ref{Exa:DifferentKappaR}).

We establish the following result, from which
Theorem~\ref{TheoremFirst} directly follows.

\begin{theorem}\label{TheoremSecond}
For each $l\in \{0,1,2\}$, there is a smooth affine surface $S$,
algebraic model of the plane $\r^2$, with trivial rational homology
groups $\tilde{H}_{i}(S_{\C};\mathbb{Q})$ and
$\kappa_\r(S)=\kappa(S)=l$. Moreover, for $l\in \{1,2\}$, we can
find infinitely many such $S$ with topologically contractible
complexifications $S_{\C}$, up to birational diffeomorphism.
\end{theorem}

As $\kappa_\r$ is invariant under birational diffeomorphisms, every
fake real plane $S$ birationally diffeomorphic to $\A^2_{\r}$
satisfies $\kappa_\r(S)=-\infty$ (Corollary~\ref{cor:Kodaira-A2}),
so every of the examples which we construct in Theorem
\ref{TheoremSecond} is a fake real plane not birationally
diffeomorphic to $\A^2_{\r}$.

\medskip

In contrast with the cases $\kappa=1,2$, the only smooth algebraic
models of $\r^2$ of Kodaira dimension $-\infty$ and $0$ with
$\mathbb{Q}$-acyclic complexifications known so far are respectively
the affine plane $\A^2_{\r}$ and a real model $Y(3,3,3)$ of one of
Fujita's exceptional surfaces \cite{Fu82} which was constructed in
\cite{DM15}. This motivates the following question:
\begin{question}
Are the surfaces $\A^2_{\r}$ and the fake real plane $Y(3,3,3)$ of
real Kodaira dimension $0$ given in $\S\ref{Kod0}$ the unique
algebraic models of $\r^2$ with trivial reduced rational homology
groups of real Kodaira dimension $-\infty$ and $0$, up to birational
diffeomorphism ?
\end{question}

The article is organised as follows: Section~\ref{Prelim} contains
some preliminaries. In  Section~\ref{Sec:RealKod}, we define the
real Kodaira dimension of a smooth real surface and establish its
basic properties. We also give some examples of fake real planes of
real Kodaira dimension $0$ (the surface $Y(3,3,3)$ in
$\S\ref{Kod0}$) and $2$ (the Ramanujam surface in
$\S\ref{Kod2Ram}$). Then, in Sections~\ref{Kod1} and \ref{Kod2}, we
provide families of pairwise not birational diffeomorphic fake real
planes of Kodaira dimension $1$ and $2$ respectively, which achieve
the proof of Theorem~\ref{TheoremSecond} hence of
Theorem~\ref{TheoremFirst}. The last subsection
($\S\ref{SubSecNontrivialforms}$) describes pairs of fake real
planes having the same complexifications but such that one has real
Kodaira dimension $2$ and the other is birationally diffeomorphic to
$\A^2_{\r}$.

We thank the referee for his careful reading and his helpful
comments to improve the exposition of this text.

\section{Preliminaries }\label{Prelim}

A $\k$-variety is a geometrically integral scheme $X$ of finite type
over a base field $\k$. A morphism of $\k$-varieties is a morphism
of $\k$-schemes. In the sequel, $\k$ will be equal to either $\r$ or
$\C$, and we will say that $X$ is a real, respectively complex,
algebraic variety. A complex algebraic variety $X$ will be said to
be defined over $\r$ if there exists a real algebraic variety
$X_{0}$ and an isomorphism of complex algebraic varieties between
$X$ and the complexification
$X{}_{0,\C}=X_{0}\times_{\mathrm{Spec}(\r)}\mathrm{\mathrm{Spec}(\C)}$
of $X_{0}$, where $\mathrm{Spec}(\C)\rightarrow\mathrm{Spec}(\r)$ is
the morphism induced by the usual inclusion
$\r\hookrightarrow\C=\r[x]/(x^{2}+1)$.

\subsection{Real algebraic varieties and morphisms between them}

For a real algebraic variety $X$, we denote by $X(\r)$ and $X(\C)$
the sets of $\r$-rational and $\C$-rational points of $X$
respectively. These are endowed in a natural way with the Euclidean
topology, locally induced by the usual Euclidean topologies on
$\A^{n}_\r(\r)\simeq\r^{2n}$ and $\A^{n}_\C(\C)\simeq\C^{n}$ respectively.
When $X$ is smooth, $X(\r)$ and $X(\C)$ can be further equipped with
natural structures of $\mathcal{C}^{\infty}$-manifolds. Every
morphism $f\colon X\rightarrow X'$ of real algebraic varieties
induces a continuous map $ X(\r)\rightarrow X'(\r)$ for the
Euclidean topologies, and an isomorphism of real algebraic varieties
$f\colon X\iso  X'$ induces a homeomorphism $X(\r)\iso X'(\r)$,
which is a diffeomorphism when $X$ and $X'$ are both smooth.

In the context of the study of real algebraic models of a
$\mathcal{C}^{\infty}$-manifold, it is natural to consider a broader
class of isomorphisms, induced by appropriate rational maps. Recall
that the domain of definition of a rational map
$\varphi:X\dashrightarrow Y$ between two $\mathbf{k}$-schemes $X$
and $Y$ is the largest open subset $\mathrm{dom}_{\varphi}$ on which
$\varphi$ is represented by a morphism. We say that $\varphi$ is
regular at a closed point $x$ if $x\in\mathrm{dom}_{\varphi}$. A
rational map $\varphi:X\dashrightarrow Y$ is called
\emph{birational} if it admits a rational inverse
$\psi:Y\dashrightarrow X$.

\begin{definition}\label{Defi:RegularRregular}{\rm
Let $\varphi\colon X\dashrightarrow X'$ be a rational map between
real algebraic varieties such that $X(\r)$ and $X'(\r)$ are not
empty.

\begin{enumerate}
\item[\rm (1)]
We say that $\varphi$ is $\r$-\emph{regular}, or that $\varphi$
induces a \emph{morphism} $X(\r)\to X'(\r)$ (that we will again
write $\varphi$), if the rational map $\varphi$ is regular at every
$\r$-rational point of $X$. Equivalently, the real locus $X(\r)$ of
$X$ is contained in the domain of definition of $\varphi$.
\item[\rm (2)]
We say that $\varphi$ is $\r$-\emph{biregular}, or that $\varphi$ is
an \emph{isomorphism} $X(\r)\iso X'(\r)$, if it is birational and
$\varphi$ and its inverse are $\r$-regular.
\item[\rm (3)]
A \emph{birational diffeomorphism} is an $\r$-biregular rational map
$\varphi\colon X\dashrightarrow X'$ between smooth real algebraic
varieties (or equivalently an isomorphism $X(\r)\iso X'(\r)$, where
$X$ and $X'$ are smooth).
\end{enumerate}}
\end{definition}

We can then consider the category most often used in real algebraic
geometry (for instance in \cite{BH,HM,BM}) whose objects are the
non-empty real loci $X(\r)$ of real algebraic varieties and whose
morphisms correspond to $\r$-regular rational
maps $X(\r)\to X'(\r)$. Note that the class of morphisms considered is in general much
larger than the class of usual regular maps. For instance, if $X$ is
a projective real algebraic surface, the group $\Aut(X)$ of
biregular automorphisms of $X$ is often quite small: its neutral
component is an algebraic group and has thus finite dimension. In
contrast, the group of birational diffeomorphisms $\Aut(X(\r))$ can
be very large. If $X$ is smooth and rational, then $\Aut(X(\r))$
acts infinitely transitively on $X(\r)$ \cite[Theorem 1.4]{HM}.
A similar behaviour can also happen if $X$ is not rational but only
geometrically rational \cite[Theorem 2]{BM}.

\subsection{Pairs and (logarithmic) Kodaira dimension}\label{Sec:PairsLogKodaira}

Recall that a \emph{Smooth Normal Crossing $($SNC$)$ divisor} $B$ on
a smooth surface $S$ defined over $\mathbf{k}$ is a curve $B$ on $S$
whose base extension $B_{\overline{\mathbf{k}}}$ to the algebraic
closure $\overline{\mathbf{k}}$ of $\mathbf{k}$ has smooth
irreducible components and ordinary double points only as
singularities. Equivalently, for every closed point $p\in
B_{\overline{\mathbf{k}}}\subseteq S_{\overline{\mathbf{k}}}$, the
local equations of the irreducible components of
$B_{\overline{\mathbf{k}}}$ passing through $p$ form a part of a
regular sequence in the maximal ideal
$\mathfrak{m}_{S_{\overline{\mathbf{k}}},p}$ of the local ring
$\mathcal{O}_{S_{\overline{\mathbf{k}}},p}$ of
$S_{\overline{\mathbf{k}}}$ at $p$.

\begin{definition}
{\rm A \emph{smooth SNC pair} $(V,B)$ is a pair consisting of a
smooth projective surface $V$ and an SNC divisor $B\subseteq V$ both
defined over $\k$.}
\end{definition}

By virtue of Nagata compactification \cite{Na62,Na63} and of
classical desingularization theorems, every smooth surface $S$
defined over $\k$ admits an open embedding $S\hookrightarrow(V,B)$
into a smooth complete (in fact, projective by virtue of Chow Lemma)
surface with possibly empty reduced SNC boundary divisor
$B=V\setminus S$, both defined over $\k$. Such a pair $(V,B)$ is
called a \emph{smooth SNC completion} of $S$.

The \emph{$($logarithmic$)$ Kodaira dimension} $\kappa(S)$ of $S$ is
then defined as the Iitaka dimension $\kappa(V,\omega_{V}(\log B))$
\cite{Ii70}, where \[\omega_{V}(\log
B)=(\det\Omega_{V/\k}^{1})\otimes\mathcal{O}_{V}(B)\simeq
\mathcal{O}_{V}(K_V+B),\] for any canonical
divisor\,$K_V$\,on\,$V$.\,More explicitly,
letting\,$R(V,B)=\bigoplus_{m\geq0}H^{0}(V,\omega_{V}(\log
B)^{\otimes m})$\,be the log-canonical ring of the smooth SNC pair
$(V,B)$, we have $\kappa(S)=\mathrm{tr}\deg_{\k}\mathcal{R}(V,B)-1$
if $H^{0}(V,\omega_{V}(\log B)^{\otimes m})\neq0$ for sufficiently
large $m$ and otherwise, if $H^{0}(V,\omega_{V}(\log B)^{\otimes
m})=0$ for every $m\geq1$, then we set by convention
$\kappa(S)=-\infty$ and we say for short that $\kappa(S)$ is
negative. The so-defined element $\kappa(S)\in\{-\infty,0,1,2\}$ is
independent of the choice of a smooth SNC completion $(V,B)$ of $S$
\cite{Ii77}, and it coincides with the usual notion of Kodaira
dimension in the case where $S$ is already complete. Furthermore, it
is invariant under arbitrary extensions of the base field $\k$, as a
consequence of the flat base change theorem
\cite[Proposition~III.9.3]{Ha77}. In particular a smooth real
surface $S$ and its complexification
$S_{\C}=S\times_{\mathrm{Spec}(\r)}\mathrm{Spec}(\C)$ have the same
Kodaira dimension.

\section{The real Kodaira dimension of open real surfaces}\label{Sec:RealKod}

\subsection{A variant of logarithmic Kodaira dimension}

For a smooth real surface $S$, the Kodaira dimension $\kappa(S)$ is
in general not a birational invariant, unless $S$ is complete: for
instance, the affine plane $\A_\r^{2}$ and the product of the punctured
affine line $\A_\r^{1}\setminus\{0\}$ with itself are
 birational to each other but have Kodaira dimensions $-\infty$
and $0$ respectively. We now introduce a variant of Kodaira
dimension more adapted to the study of equivalence classes of open
real surfaces up to birational diffeomorphisms.

\begin{notation}{\rm
Given a smooth SNC pair $(V,B)$ defined over $\r$, we denote by
$B_{\r}\subseteq B$ the union of all irreducible components $B_i$ of
$B_{\mathbb{C}}$ which are defined over $\r$ and such that
$B_{i}(\r)$ is infinite.}
\end{notation}

\begin{remark}\label{Rem:ZariskiClosureBr}{\rm
The Zariski closure of $B(\r)$ in $V$ is the union of $B_{\r}$ and
of finitely many isolated points of $B$.}
\end{remark}

\begin{definition}{\rm
Let $(V,B)$ be a smooth SNC pair defined over $\r$. We say that
$B_{\r}$ contains an \emph{imaginary loop} if there exists a pair of
distinct irreducible components $A$ and $A'$ of $B_{\mathbb{C}}$
defined over $\r$ and with infinite real loci, whose intersection
$A\cap A'$ contains a pair of conjugate non-real points, i.e. a $\mathbb{C}$-rational but not $\r$-rational point.}
\end{definition}

\begin{definition}\label{Defi:KappaR}\label{def:kappaR}{\rm
The \emph{real Kodaira dimension} of a smooth SNC pair $(V,B)$
defined over $\r$ is the element
\[\kappa_{\r}(V,B)=\kappa(V,\omega_{V}(\log B_{\r}))\in \{-\infty,
0,1,2\}.\]

\noindent If furthermore $B_{\r}$ has no imaginary loop, then we
define the real Kodaira dimension of $S=V\setminus B$ to be
\[\kappa_{\r}(S)=\kappa(S(\r))=\kappa_{\r}(V,B).\]}
\end{definition}

By definition, given a smooth SNC pair  $(V,B)$ defined over $\r$,
the curve $B_\r$ contains imaginary loops if and only if it has some
pairs of non-real singular points $q$ and $\overline{q}$. The following lemma provides
a simple procedure to eliminate imaginary loops.

\begin{lemma} \label{Rem:Imagloops}
Let $(V,B)$ be a smooth SNC pair defined over $\r$ and let
$Z=\{q_1,\overline{q}_1,\ldots, q_s,\overline{q}_s\}$ be the set of
non-real singular points of $B_{\r}$. Let $\tau\colon \hat{V}\to V$
be the blow-up of $Z$ and let $E=\sum_{i=1}^{s} \tau^{-1}(q_i)+
\tau^{-1}(\overline{q}_i)$ be its exceptional locus. Then the
following hold:
\begin{enumerate}\item[\rm (1)]\phantomsection\label{Imagloop1}
  $(\hat{V},\hat{B}=\tau^*(B)_{\mathrm{red}})$ is a smooth SNC pair defined over $\r$ for which $\hat{B}_\r$ has no imaginary loops and such that  $\tau$ induces an isomorphism $\hat{V}\setminus \hat{B}\to V\setminus B$.
\item[\rm (2)]\phantomsection\label{Imagloop2}
$ \kappa(\hat{V},\hat{B})=\kappa(V,B)$ and  $
\kappa_{\r}(\hat{V},\hat{B})\le \kappa_{\r}(V,B)$.
  \end{enumerate}
\end{lemma}

\begin{proof}
\hyperref[Imagloop1]{(1)}: As $B$ is SNC, the morphism $\tau$ only
blows-up ordinary double points of $B$, so $\hat{B}$ is again SNC.
Every irreducible curve on $\hat{V}_\C$  contracted by $\tau$ is not defined
over $\r$ and does not intersect its conjugate, so does not contain
any real point. This implies that $\hat{B}_{\r}$ is the strict
transform of $B_{\r}$. Every singular $\C$-rational point of $B_{\r}$ which
was not real has been blown-up, and every singular $\C$-rational point of
$\hat{B}_{\r}$ is an $\r$-rational point. Hence, $\hat{B}_{\r}$ has
no imaginary loop. The fact that $\tau$ induces an isomorphism
$\hat{V}\setminus \hat{B}\to V\setminus B$ follows from the fact
that $\hat{B}=\tau^*(B)_{\mathrm{red}}$ and that all points blown-up
by $\tau$ and all exceptional divisors of $E$ are contained in $B$
and $\hat{B}$ respectively.

\hyperref[Imagloop2]{(2)}: Since the points blown-up by $\tau$ are
ordinary double points of $B_{\r}$ hence of $B$, we have
$\hat{B}=\tau^*B-E$ whereas $\hat{B}_{\r}=\tau^*B_{\r}-2E$ because
$E$ does not contain any real point. Denoting by $K_{\hat{V}}$ and
$K_V$ the canonical divisors on $\hat{V}$ and $V$ respectively, we
have the ramification formula $K_{\hat{V}}=\tau^*K_V+E$  for $\tau$.
This yields the two equalities
\[K_{\hat{V}}+\hat{B}=\tau^*(K_V+B)\text{ and }K_{\hat{V}}+\hat{B}_{\r}=\tau^*(K_V+B_{\r})-E.\]
The first equality gives $ \kappa(\hat{V},\hat{B})=\kappa(V,B)$. The
second equality gives $ \kappa_{\r}(\hat{V},\hat{B})\le
\kappa_{\r}(V,B)$, since $E$ is effective. \qed
\end{proof}

The following example shows that the inequality of
Lemma~\ref{Rem:Imagloops}\,\hyperref[Imagloop2]{(2)} can be strict.

\begin{example}\label{Exa:DifferentKappaR}{\rm
Take $V=\p^2_\r$ and $B=L+C$, where $L\simeq \p^1_\r$ is the line of equation $x=0$
and $C$ is the smooth conic of equation $x^2-y^2-z^2=0$. Then,
$B_\r=B$ has imaginary loops, as the points $q=[0:1:i]$,
$\overline{q}=[0:1:-i]$ are singular points of $B$. With the
notation of Lemma~\ref{Rem:Imagloops}, the blow-up $\tau\colon
\hat{V}\to \p^2_\r$ of these two points yields an SNC pair
$(\hat{V},\hat{B}=\tau^*(B)_{\mathrm{red}})$ such that
$\hat{B}=\tilde{L}+\tilde{C}+E$, where  $E=E_q+E_{\overline{q}}$ is
the sum of the exceptional divisors over $q$ and $\overline{q}$
respectively, and where $\tilde{L}$, $\tilde{C}$ are the strict
transforms of $L$ and $C$. We get $\hat{B}_\r=\tilde{L}+\tilde{C}$.

The canonical divisor of $V=\p^2_\r$ satisfies $K_V= -L-C$ so that
$K_V+B_{\r}=0$. On the other hand,  by the proof of
Lemma~\ref{Rem:Imagloops}\,\hyperref[Imagloop2]{(2)}, we have
$K_{\hat{V}}+\hat{B}_\r=\tau^*(K_V+B_{\r})-E=-E$. Hence
$\kappa_\r(V,B)=0$ whereas $\kappa_\r(\hat{V},\hat{B})=-\infty$.}
\end{example}

\medskip
The aim of this section is to show that the definition of
$\kappa(S(\r))$ (or $\kappa_{\r}(S)$) only depends on the
birationnal diffeomorphism  class of $S(\r)$, or equivalently of the
real surface $S$, up to birational diffeomorphism.

\medskip

The following notion is natural to compare two possible pairs, up to
birational diffeomorphism.

\begin{definition}\label{Defi:BirPairs}{\rm
Let $(V,B)$ and $(V',B')$ be two smooth SNC pairs defined over $\r$.
A \emph{birational map of pairs} $\varphi\colon (V,B)\dashrightarrow
(V',B')$ is a birational map $V\dashrightarrow V'$ defined over $\r$
inducing an isomorphism \[(V\setminus B)(\r)\iso (V'\setminus
B')(\r)\] (or equivalently inducing a birational diffeomorphism from
$V\setminus B$ to $V'\setminus B'$).}
\end{definition}

\begin{example}\label{Exa:Links}{\rm
Let $(V,B)$ and $(V',B')$ be two smooth SNC pairs defined over $\r$,
and let $\tau\colon V\to V'$ be a birational morphism, defined over
$\r$. In each of the following cases, $\tau$ yields a birational map
of pairs $(V,B)\dashrightarrow (V',B')$.
\begin{enumerate}
\item[\rm (1)] \phantomsection\label{ExaIso}
If $\tau$ is an isomorphism $V\iso V'$ such that
$\varphi(B(\r))=B'(\r)$.
\item[\rm (2)] \phantomsection\label{ExaReal}
If $\tau$ is the blow-up of a point $q\in B'(\r)$ and
$B=\varphi^{-1}(B')_{\mathrm{red}}$.
\item[\rm (3)] \phantomsection\label{ExaNonreal}
If $\tau$ is the blow-up of a pair of conjugate non-real points
$q,\overline{q}\in V'(\C)$ and $B$ is the strict transform of $B'$.
(Here the exceptional locus does not contain any real point.)
\end{enumerate}}
\end{example}

\begin{remark}{\rm
Another example of simple birational map of pairs $\tau\colon
(V,B)\to (V',B')$ is given as follows: we take $\tau$ to be the
blow-up of a pair of conjugate non-real points $q,\overline{q}\in
V'(\C)$ and $B=\varphi^{-1}(B')_{\mathrm{red}}$.

Denoting by $E\subseteq V$ the exceptional locus of $\tau$ (which is
the disjoint union of two conjugate imaginary $(-1)$-curves and does
not contain any real point) and by $\tilde{B}$ the strict transform
of $B'$, we get $B=\tilde{B}+E$. We can then decompose the
birational map $\tau\colon (V,B)\to (V',B')$ as the composition of
$\mathrm{id}_V\colon (V,B)\to (V,\tilde{B})$ with the birational
morphism $\tau\colon (V,\tilde{B})\to (V',B')$, which are examples
of type \hyperref[ExaIso]{(1)} and \hyperref[ExaNonreal]{(3)}
respectively.}
\end{remark}

\begin{lemma}\label{Lem:Decomposition}
Let $\varphi\colon (V,B)\dashrightarrow (V',B')$ be a birational map
of smooth SNC pairs. Then, there exists a sequence of birational
maps of pairs
\[(V,B)=(V_0,B_0)\stackrel{\varphi_1}{\dashrightarrow}(V_1,B_1)\stackrel{\varphi_2}{\dashrightarrow}\ \cdots \ \stackrel{\varphi_{n-1}}{\dashrightarrow} (V_{n-1},B_{n-1})\stackrel{\varphi_n}{\dashrightarrow} (V_n,B_n)=(V',B')\]
such that $\varphi=\varphi_n\circ\dots \circ \varphi_1$ and such
that for each $i\in \{1,\dots,n\}$, either $\varphi_i$ or
$(\varphi_i)^{-1}$ is of one of three types
\hyperref[ExaIso]{(1)}-\hyperref[ExaReal]{(2)}-\hyperref[ExaNonreal]{(3)}
of Example~$\ref{Exa:Links}$.

Moreover, if $B_\r$ and $B'_\r$ have no imaginary loop, then we can
assume the same for $(B_i)_\r$, for $i=1,\dots,n$.
\end{lemma}
\begin{proof}
By definition, $\varphi\colon V\dasharrow V'$ is a birational map
defined over $\r$, inducing an isomorphism between $(V\setminus B)(\r)$ 
and $(V'\setminus B')(\r)$.

If $\varphi$ is an isomorphism $V\iso V'$, then it sends $B(\r)$
onto $B'(\r)$ and is thus of the type of
Example~\ref{Exa:Links}\,\hyperref[ExaIso]{(1)}. Otherwise, we can
take a minimal resolution of the indeterminacies of $\varphi$ given
by
\[\xymatrix@R=3.5mm{
      & W\ar[dl]_{\tau}\ar[dr]^{\tau'} \\
    V \ar@{-->}[rr]^{\varphi} && V' \\
    }
\]
where $W$ is a smooth projective real surface and $\tau$ and $\tau'$
are birational morphisms defined over $\r$. Recall that since
$\varphi$ and $\varphi^{-1}$ are  defined over $\r$, the union of
their base-points, including infinitely near ones, is defined over
$\r$, hence consists of either real points or pair of conjugate
non-real points. The minimality assumption implies in particular
that $\tau$ and $\tau'$ are the blow-ups of the base-points of
$\varphi$ and $\varphi^{-1}$ respectively. This gives back the
classical decomposition of $\tau$ and $\tau'$ into simple blow-ups
(one real or a pair of conjugate non-real points) and thus the real
Zariski strong factorisation of $\varphi$, as explained for
instance in \cite[Chapter~II, Proposition~6.4]{Sil89}.

We proceed by induction on the number of such points, the case where
there is no base-point being
\ref{Exa:Links}\,\hyperref[ExaIso]{(1)}.

If $q\in V(\r)$ is a base-point of $\varphi$, then $q$ belongs to
$B(\r)$, since $\varphi$ induces an isomorphism between $(V\setminus
B)(\r)$ and $(V'\setminus B')(\r)$. We can write $\tau$ as
$\tau=\tau_q\circ \hat{\tau}$, where $\tau_q\colon \hat{V}\to V$ is
the blow-up of $q$ and $\hat{\tau}\colon W\to \hat{V}$ is a
birational morphism defined over $\r$. Writing
$\hat{B}=(\tau_q)^{-1}(B)_{\mathrm{red}}$, the birational map
$\tau_q$ yields a birational maps of pairs
$(\hat{V},\hat{B})\dasharrow (V,B)$ of type
\ref{Exa:Links}\,\hyperref[ExaReal]{(2)}. Moreover, if $B_\r$ has no
imaginary loop, the same holds for
$\hat{B}_\r=((\tau_q)^{-1}(B_\r))_{\mathrm{red}}$.  As
$\varphi\circ\tau_q\colon (\hat{V},\hat{B})\dasharrow (V',B')$ is
again a birational maps of pairs, whose minimal resolution has less
base-points, we conclude by induction.

The same argument works with a point $q\in V'(\r)$ which is a
base-point of $\varphi^{-1}$. We can thus assume that no point of
$V(\r)$ or $V'(\r)$ is a base-point of $\varphi$ or $\varphi^{-1}$.

If $\tau$ is not an isomorphism, there is a pair of conjugate
non-real points $q,\overline{q}\in V(\C)$, both base-points of
$\varphi$, blown-up by $\tau$. As before, we write $\tau$ as
$\tau=\tau_q\circ \hat{\tau}$, where $\tau_q\colon \hat{V}\to V$ is
the blow-up of $q$ and $\overline{q}$, which is thus defined over
$\r$. Then $\hat{\tau}\colon W\to \hat{V}$ is a birational morphism
defined over $\r$. The strict transform $\hat{B}$ of $B$ on
$\hat{V}$ is then again an SNC-divisor, defined over $\r$, and
$\tau_q$ induces a birational map of pairs
$(\hat{V},\hat{B})\dasharrow (V,B)$ of type
\ref{Exa:Links}\,\hyperref[ExaNonreal]{(3)}. Moreover, if $B_\r$ has
no imaginary loop, the same holds for $\hat{B}_\r$, which is the
strict transform of $B_\r$. As before, the result follows by
induction. The same works when $\tau'$ is not a regular morphism.
\qed
\end{proof}

\begin{lemma}
\label{lem:Invariance-kappa} Let $\varphi\colon
(V,B)\dashrightarrow(V',B')$ be a birational map of smooth SNC pairs
$(V,B)$ and $(V',B')$ defined over $\r$, such that neither $B_\r$
nor $B'_\r$ has an imaginary loop. Then for every $m,n\in
\mathbb{Z}$ with $m\ge \lvert n\rvert$, the map $\varphi$ induces an
isomorphism
\[\varphi_{*}\colon H^{0}(V,m K_{V}+n B_{\r})\iso H^{0}(V',m K_{V'}+n B'_{\r}).\]
In particular, $\kappa_{\r}(V,B)=\kappa_{\r}(V',B')$.\end{lemma}

\begin{proof}
Applying Lemma~\ref{Lem:Decomposition}, we can assume that $\varphi$
is of one of the three cases
\hyperref[ExaIso]{(1)}-\hyperref[ExaReal]{(2)}-\hyperref[ExaNonreal]{(3)}
of Example~$\ref{Exa:Links}$.

In case \hyperref[ExaIso]{(1)}, $\varphi$ is an isomorphism $V\iso
V'$ that sends $B(\r)$ isomorphically onto $B'(\r)$. It  thus maps
the Zariski closure of $B(\r)$ isomorphically onto that of $B'(\r)$.
This implies that $\varphi(B_\r)=B'_\r$ (see
Remark~\ref{Rem:ZariskiClosureBr}). This achieves the proof in this
case.

We then do the two cases
\hyperref[ExaReal]{(2)}-\hyperref[ExaNonreal]{(3)}, and denote, in
both cases, by $E\subseteq V$ the divisor contracted by $\varphi$.

In case \hyperref[ExaReal]{(2)}, $\varphi$ is the blow-up of a point
$q\in B'(\r)$, $B=\varphi^{-1}(B')_{\mathrm{red}}$ and
$E=\varphi^{-1}(q)$.

 In case \hyperref[ExaNonreal]{(3)}, $\varphi$ is the blow-up of a pair of conjugate non-real points $q,\overline{q}\in V'(\C)$, $B$ is the strict transform of $B'$ and $E=\varphi^{-1}(q)+\varphi^{-1}(\overline{q})$. As $B'_\r$ is an SNC-divisor with no imaginary loop, the points $q,\overline{q}$ cannot be singular points of $B'_\r$.

 We find respectively
\[\begin{array}{llll}
\text{\hyperref[ExaReal]{(2)}:}& B_{\r}&=&\begin{cases}
\varphi^{*}(B'_{\r})+E & \textrm{if }q\in B'(\r)\setminus B'_{\r}(\r),\\
\varphi^{*}(B'_{\r}) & \textrm{if }q\textrm{ is a simple point of }B'_{\r},\\
\varphi^{*}(B'_{\r})-E & \textrm{if }q\textrm{ is a double point of
}B'_{\r},
\end{cases} \vspace{0.1cm}\\
\text{\hyperref[ExaNonreal]{(3)}:}& B_{\r}&=&\begin{cases}
\varphi^{*}(B'_{\r}) & \textrm{if }q\notin B'_{\r}(\C),\\
\varphi^{*}(B'_{\r})-E & \textrm{if }q\in B'_{\r}(\C).
\end{cases}\end{array}\]
Since $K_V=\varphi^{*}(K_{V'})+E$, we obtain
\[m K_V+nB_{\r}=m(\varphi^{*}(K_{V'})+E)+n(\varphi^{*}(B'_\r)+\epsilon E) =\varphi^{*}(m K_{V'}+ n B'_\r) +\delta E,\]
where $\epsilon\in \{-1,0,1\}$ and $\delta=m+n\epsilon \ge m-\lvert
n\rvert\ge  0$ as $m\ge \lvert n\rvert$ by hypothesis. As a
consequence, the natural inclusion
\[
H^{0}(V',m K_{V'}+n B'_{\r})\simeq H^{0}(V,\varphi^{*}(m K_{V'}+n
B'_{\r}))\hookrightarrow H^{0}(V,\varphi^{*}(m K_{V'}+ n B'_\r)
+\delta E)
\]
is a bijection.

Indeed, for each integer $r\ge 0$, an effective divisor $D$
equivalent to  $\varphi^{*}(m K_{V'}+ n B'_\r) +r E$ is equal to
$\tilde{D}+r E$ for some effective divisor $\tilde{D}$ equivalent to
$\varphi^{*}(m K_{V'}+ n B'_\r)$. This is clear for $r=0$, and for
$r>0$ we just compute $E\cdot D=r E^2\le -r<0$ and obtain that $D-E$
is effective, which yields the result by induction.

The case where $m=n$ shows that
$\kappa_{\r}(V',B')=\kappa_{\r}(V,B)$. \qed
\end{proof}

As a consequence of Lemma~\ref{lem:Invariance-kappa}, we obtain:

\begin{corollary}\label{RealKodIsIvariant}
For a smooth real affine surface $S$, $\kappa_{\r}(V,B)$ is
independent of the choice of a smooth SNC completion $(V,B)$ of $S$
defined over $\r$ and such that $B_\r$ does not have an imaginary
loop. The real Kodaira dimension $\kappa_{\r}(S)\in
\{-\infty,0,1,2\}$ of $S$ introduced in
Definition~$\ref{Defi:KappaR}$ is thus a well-defined invariant of
$S(\r)$.
\end{corollary}

The following result summarises immediate consequences of the
definition and Lemma \ref{lem:Invariance-kappa}.

\begin{proposition}
\label{Prop:KappaR-properties} The real Kodaira dimension
$\kappa_{\r}(S)$ of a smooth real surface $S$ enjoys the following
properties:
\begin{enumerate}
\item[\rm (1)] \phantomsection\label{kapparInvariant}
$\kappa_{\r}(S)=\kappa_{\r}(S')$ for every smooth surface $S'$
defined over $\r$ birationally diffeomorphic to $S$.
\item[\rm (2)] \phantomsection\label{kapparsmallerkappa}
$\kappa_{\r}(S)\leq\kappa(S)$ with equality if $S$ admits a smooth
projective SNC completion $(V,B)$ defined over $\r$ such that
$B=B_\r$ has no imaginary loop.
\end{enumerate}
\end{proposition}

\begin{proof}
Assertion \hyperref[kapparInvariant]{(1)} follows from
Lemma~\ref{lem:Invariance-kappa}. For \hyperref[kapparsmallerkappa]{(2)}, it follows from Lemma ~\ref{Rem:Imagloops} that the surface $S$ always admits a smooth projective SNC completion $(V,B)$ defined over $\r$ such that
$B_{\r}$ has no imaginary loop. We then have $B=B_\r +E$ for some effective divisor $E$, and thus get
$\kappa_\r(S)=\kappa_\r(V,B)=\kappa(V,B_\r)\le\kappa(V,B)=\kappa(S)$,
with equality if $B=B_\r$. \qed
\end{proof}

\begin{example}{\rm
The inequality $\kappa_{\r}(S)\leq\kappa(S)$ is strict in general:
for instance, let $B\subseteq\mathbb{P}_{\r}^{2}$ be a general
arrangement consisting of $0\leq r\leq2$ real lines and a collection
of $p\geq0$ pairs of non-real complex conjugate lines. Then for
$S=\mathbb{P}_{\r}^{2}\setminus B$, we have $\kappa_{\r}(S)=-\infty$
independently of $r$ and $p$ while

\[
\kappa(S)=\begin{cases}
-\infty & \textrm{if }r+2p<3,\\
0 & \textrm{if }r+2p=3,\\
2 & \textrm{if }r+2p\geq4.
\end{cases}
\]
The equality $\kappa_{\r}(S)=-\infty$ follows from the fact that $S$
is birationally diffeomorphic to the complement $S'$ of $r\le 2$
lines in $\p^2_{\r}$, which satisfies
$\kappa_{\r}(S')=\kappa(S)=-\infty$. On the other hand, since $B$ is
an SNC divisor, $\kappa(S)=\kappa(\p^2_{\r},B)$ where
$K_{\p^2_{\r}}+B$ has degree $-3+r+2p$.}
\end{example}

As a consequence of
Proposition~\ref{Prop:KappaR-properties}\,\hyperref[kapparInvariant]{(1)},
we obtain:

\begin{corollary}
\label{cor:Kodaira-A2} Let $S$ be a smooth real surface. If $S$ is
birationally diffeomorphic to $\A_{\r}^{2}$, then
$\kappa_{\r}(S)=-\infty$.
\end{corollary}

\begin{proof}
Follows from
Proposition~\ref{Prop:KappaR-properties}\,\hyperref[kapparInvariant]{(1)},
and the fact that $\A^2_{\r}=\p^2_{\r}\setminus L$, where
$L\subseteq \p^2_{\r}$ is a real line. Hence,
$m(K_{\p^2_{\r}}+L)\simeq -2mL$ is not effective for each $m\ge 0$,
so
$\kappa_{\r}(S)=\kappa_{\r}(\A^2_\r)=\kappa(\p^2_{\r},L)=-\infty$.
\qed
\end{proof}

\subsection{Examples}

\subsubsection{ An algebraic model of real Kodaira dimension $0$: the exceptional fake
plane $Y(3,3,3)$} \label{Kod0}

 Let us recall from \cite[\S 5.1.1]{DM15} the following construction of a fake  plane
$S$ of Kodaira dimension $0$ whose complexification $S_{\C}$ is
$\mathbb{Q}$-acyclic\footnote{It is known that there is no fake real
plane of Kodaira dimension $0$ with $\mathbb{Z}$-acyclic
complexification \cite[Theorem 4.7.1(1), p.~244]{MiyBook}}, with
$H_{1}(S_{\C};\mathbb{Z})\simeq\mathbb{Z}_{9}$. Let $D$ be the union
of four general real lines $\ell_{i}\simeq\mathbb{P}_{\r}^{1}$,
$i=0,1,2,3$ in $\mathbb{P}^2_{\r}$ and let
$\tau:V\rightarrow\mathbb{P}^{2}_\r$ be the real projective surface
obtained by first blowing-up the real points
$p_{ij}=\ell_{i}\cap\ell_{j}$ with exceptional divisors $E_{ij}$,
$i,j=1,2,3$, $i\neq j$ and then blowing-up the real points
$\ell_{1}\cap E_{12}$, $\ell_{2}\cap E_{23}$ and $\ell_{3}\cap
E_{13}$ with respective exceptional divisors $E_{1}$, $E_{2}$ and
$E_{3}$. We let $B=\ell_{0}\cup\ell_{1}\cup\ell_{2}\cup\ell_{3}\cup
E_{12}\cup E_{23}\cup E_{13}$. The dual graphs of $D$, of its total
transform $\tau^{-1}(D)$ in $V$ and of $B$ are depicted in Figure
\hyperref[fig:logkod0-Y333]{1}.

\begin{figure}[!htbp]
\centerline{\input{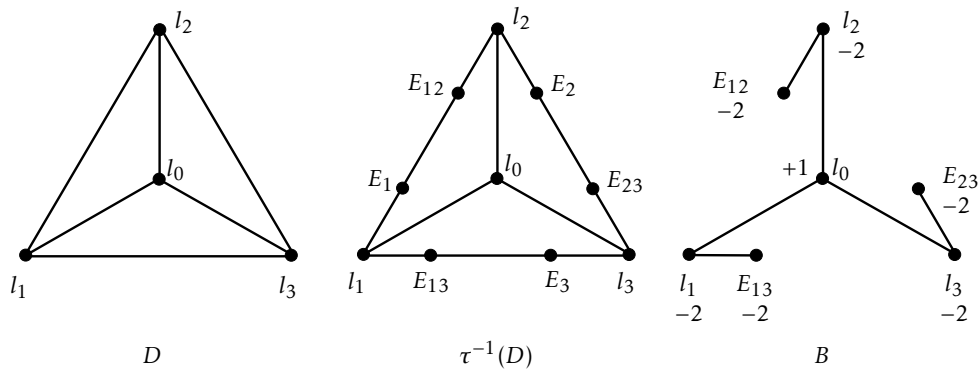}} \caption{Construction of
$Y(3,3,3)$}
\vspace{-6cm}\phantomsection\label{fig:logkod0-Y333}\vspace{6cm}
\end{figure}

By virtue of \cite[\S 5.1.1]{DM15} (see also \cite[\S
3.2]{DMSpitz}), the real surface $Y(3,3,3)=V\setminus B$ is a fake
real plane of Kodaira dimension $\kappa(Y(3,3,3))=0$. Since by
construction $B=B_{\r}$ is a tree, we conclude by
Proposition~\ref{Prop:KappaR-properties}\,\hyperref[kapparInvariant]{(1)}
that $\kappa_{\r}(Y(3,3,3))=\kappa(Y(3,3,3))=0$, hence that
$Y(3,3,3)$ is not birationally diffeomorphic to $\A_{\r}^{2}$. This
answers \cite[Question 5.2]{DM15}.

\subsubsection{An algebraic model of real Kodaira dimension $2$: the real Ramanujam surface}\label{Kod2Ram}

The real Ramanujam surface $S$ is a real model of the complex
Ramanujam surface \cite{Ram71} which is constructed as follows: let
$D\subseteq\mathbb{P}_{\r}^{2}=\mathrm{Proj}(\r[x,y,z])$ be the
union of the cuspidal cubic $C=\{x^{2}z+y^{3}=0\}$ with its
osculating conic $Q$ at an $\r$-rational point $q\in C(\r)$ distinct
from the singular point $[0:0:1]$ of $C$ and its flex $[0:0:1]$. Up
to change of coordinates, one can for instance choose $q=[1:1:-1]$,
which implies that the equation of $Q$ is
\[5x^2+24xy-40xz+45y^2-15yz-z^2.\]
So $Q$ is a smooth $\r$-rational conic intersecting $C$ at $q$ with
multiplicity $5$ and transversally at a second $\r$-rational point
$p$. We let $\beta:\mathbb{F}_{1}\rightarrow\mathbb{P}_{\r}^{2}$ be
the blow-up of $p$ with exceptional divisor
$E\simeq\mathbb{P}_{\r}^{1}$ and we let $S$ be the complement in
$\mathbb{F}_{1}$ of the proper transform $\tilde{D}$ of $D$. The
total transform $B$ of $\tilde{D}$ in a minimal log-resolution
$\tau:(V,B)\rightarrow(\mathbb{F}_{1},\tilde{D})$ of the pair
$(\mathbb{F}_{1},\tilde{D})$ is a tree of $\r$-rational curves
depicted in Figure \hyperref[fig:Ramanujam-reso]{2}.

\begin{figure}[ht]
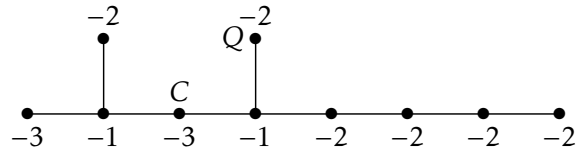

\xy (-85,5)*{}; (-40,0)*{};(30,0)*{}**\dir{-};
(-30,0)*{};(-30,10)*{}**\dir{-}; (-10,0)*{};(-10,10)*{}**\dir{-};
(-40,0)*{\bullet};  (-30,0)*{\bullet}; (-30,10)*{\bullet};
(-20,0)*{\bullet};
 (-10,0)*{\bullet}; (-10,10)*{\bullet};
(0,0)*{\bullet};  (10,0)*{\bullet};
 (20,0)*{\bullet};   (30,0)*{\bullet};
(-40,-3)*{-3}; (-30,-3)*{-1}; (-30,13)*{-2};
(-20,3)*{C};(-13,10)*{Q};  (-20,-3)*{-3}; (-10,-3)*{-1};
(-10,13)*{-2}; (0,-3)*{-2};
 (10,-3)*{-2};
 (20,-3)*{-2};
 (30,-3)*{-2};
\endxy
\caption{The weighted dual graph of the divisor $B\subseteq V$.}
\vspace{-3cm}\phantomsection\label{fig:Ramanujam-reso}\vspace{3cm}
\end{figure}

The surface $S$ is a fake real plane of Kodaira dimension $2$ with
contractible complexification $S_{\mathbb{C}}$: the contractibility
of $S_{\mathbb{C}}$ was first established by Ramanujam \cite{Ram71},
the fact that $\kappa(S)=2$ follows for instance from the
classification of contractible complex surfaces of Kodaira dimension
$\leq1$ established in \cite{GuMi87} (see also \cite{Ii78}), and the
fact that $S(\mathbb{R})\simeq \mathbb{R}^2$ was proven in
\cite[Example 3.8]{DM15}.

Since the smooth SNC completion $(V,B)$ of $S$ has the property that
$B_{\r}=B$ is a tree, the equality $\kappa_{\r}(S)=\kappa(S)=2$ holds by virtue
of
Proposition~\ref{Prop:KappaR-properties}\,\hyperref[kapparInvariant]{(1)},
and so, $S$ is not birationally diffeomorphic to $\A_{\r}^{2}$.

\begin{remark}{\rm
The same argument as above also applies to the three examples of
fake real planes $S$ of log-general type with contractible
complexification $S_{\C}$ constructed in \cite[\S 5.1]{DMSpitz} from
arrangements of real lines and irreducible singular $\r$-rational
quartics in $\mathbb{P}_{\r}^{2}$: all these surfaces have the
property to admit a smooth SNC completion $(V,B)$ defined over $\r$
for which $B_{\r}=B$ is a tree, so that their real Kodaira dimension
$\kappa_{\r}$ coincides with their usual Kodaira dimension. All of
them are therefore non birationally diffeomorphic to~$\A_{\r}^{2}$.}
\end{remark}

\section{Families of algebraic models of Kodaira dimension $1$ }\label{Kod1}

Fake real planes $S$ of Kodaira dimension $1$ whose
complexifications $S_{\mathbb{C}}$ are $\mathbb{Z}$-acyclic
manifolds, that is topological manifolds with trivial reduced
homology groups $\tilde{H}_i(S_{\mathbb{C}};\mathbb{Z})$, have been
classified up to isomorphism in \cite{DM15} (see also \cite{GuMi87}
and \cite{tDPe90} for the complex case). One obtains the following:

\begin{proposition}
\label{thm:Fake-Kodaira-1} A fake real plane of Kodaira dimension
$\kappa=1$ with $\mathbb{Z}$-acyclic complexification is not
birationally diffeomorphic to $\mathbb{A}_{\mathbb{R}}^{2}$.
\end{proposition}

\begin{proof}
 By virtue of \cite[Theorem 3.2]{DM15}, every such
surface admits a completion into a smooth projective surface $V$
defined over $\mathbb{R}$ obtained from $\mathbb{P}^{2}_{\r}$ by
blowing-up specific sequences of real points, and whose boundary
$B=V\setminus S$ consists of a tree of
$\mathbb{P}_{\mathbb{R}}^{1}$'s. In particular, such a smooth pair
$(V,B)$ satisfies $B_{\mathbb{R}}=B$, and we deduce from Proposition
\ref{Prop:KappaR-properties}\,\hyperref[kapparInvariant]{(1)} that
$\kappa_{\mathbb{R}}(S)=\kappa(S)=1$. The result then follows from
Corollary~\ref{cor:Kodaira-A2}. \qed
\end{proof}

In the rest of this section, we build on a blow-up construction of
certain fake real planes of Kodaira dimension $1$ with contractible
complexifications \cite[Example 3.5]{DM15} to derive the existence
of infinitely many pairwise non birationally diffeomorphic such
surfaces. The main ingredient is the uniqueness of the log-canonical
fibration, given by Lemma~\ref{lem:Invariance-kappa}.

\vspace{1em}

Let $1<a<b$ be a pair of coprime integers and consider the rational
pencil
\[
\Psi:\mathbb{P}_{\mathbb{R}}^{2}=\mathrm{Proj}_{\mathbb{R}}(\mathbb{R}[x,y,z])\dashrightarrow\mathbb{P}_{\mathbb{R}}^{1},\quad[x:y:z]\mapsto[y^{b}:x^{a}z^{b-a}].
\]
It has two proper base points $q_{0}=[0:0:1]$ and
$q_{\infty}=[1:0:0]$. A general geometrically irreducible fiber of
$\Psi$ is an $\mathbb{R}$-rational cuspidal curve, with multiplicity
$a$ and $b-a$ at $q_0$ and $q_\infty$ respectively, and $\Psi$ has
precisely two degenerate members: $\Psi^{-1}([1:0])$ which is
supported on the union of the lines $L_{x}=\{x=0\}$ and
$L_{z}=\{z=0\}$ and $\Psi^{-1}([0:1])$ which is supported on the
line $L_{y}=\{y=0\}$. Up to exchanging the roles of $x$ and $z$, we
assume from now on that $a>b-a$.

Let $C_{a,b}=\Psi^{-1}([1:-1])=\{x^{a}z^{b-a}-y^{b}=0\}$, let
$p=[1:1:1]\in C_{a,b}$ and let
$\beta:X(a,b)\rightarrow\mathbb{P}_{\mathbb{R}}^{2}$ be the blow-up
of $p$, with exceptional divisor $E$. We let
$S(a,b)=X(a,b)\setminus(C_{a,b}\cup L_{z})$ where we identified a
curve in $\mathbb{P}_{\mathbb{R}}^{2}$ with its proper transform in
$X(a,b)$.

The dual graph of the total transform of $C_{a,b}\cup L_x \cup L_y \cup L_z$ in the
minimal resolution $\alpha:V(a,b)\rightarrow X(a,b)$ of the induced
rational map
$\Psi\circ\beta:X(a,b)\dasharrow\mathbb{P}_{\mathbb{R}}^{1}$ is
depicted in Figure \hyperref[fig:logkod1EX]{3}. The boundary
$B(a,b)=V(a,b)\setminus S(a,b)$ is the reduced total transform of
$C_{a,b}\cup L_{z}$. The induced morphism
$f=\Psi\circ\beta\circ\alpha:V(a,b)\rightarrow\mathbb{P}_{\mathbb{R}}^{1}$
is a $\mathbb{P}^{1}$-fibration having the last exceptional divisors
$C_{0}$ and $C_{1}$ of $\alpha$ over the points $q_{0}$ and
$q_{\infty}$ as disjoint sections.

\begin{figure}[ht]
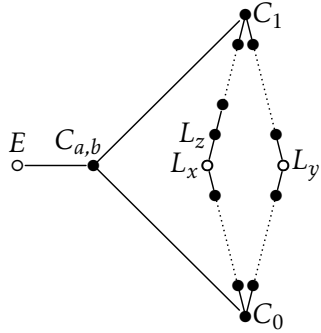

\xy (-95,5)*{}; (-40,0)*{\circ};(-30,0)*{\bullet}**\dir{-};
(-30,0)*{};(-10,20)*{\bullet}**\dir{-};
(-30,0)*{};(-10,-20)*{\bullet}**\dir{-};
(-10,20)*{};(-11,16)*{\bullet}**\dir{-};(-11,16)*{};(-13,8)*{}**\dir{..};(-13,8)*{\bullet};(-15,0)*{\circ}**\dir{-};
(-15,0)*{\circ};(-14,-4)*{\bullet}**\dir{-};(-14,-4)*{};(-11,-16)*{\bullet}**\dir{..};(-11,-16)*{};(-10,-20)*{}**\dir{-};
(-10,20)*{};(-9,16)*{\bullet}**\dir{-};(-9,16)*{};(-6,4)*{\bullet}**\dir{..};(-6,4)*{};(-5,0)*{\circ}**\dir{-};
(-5,0)*{\circ};(-6,-4)*{\bullet}**\dir{-};(-6,-4)*{};(-9,-16)*{\bullet}**\dir{..};(-9,-16)*{};(-10,-20)*{}**\dir{-};
(-14,4)*{\bullet}; (-40,3)*{E};(-32,3)*{C_{a,b}};
(-18,0)*{L_x};(-17,4)*{L_z};
(-2,0)*{L_y};(-7,-20)*{C_{0}};(-7,20)*{C_{1}};
\endxy
\caption{The dual graph of the total transform of $C_{a,b}\cup L_x
\cup L_y \cup L_z$, the components denoted by $\circ$ are those
which do not belong to the boundary $B(a,b)$. }
\vspace{-5.8cm}\phantomsection\label{fig:logkod1EX}\vspace{5.8cm}
\end{figure}

\begin{proposition} \label{prop:fakeCont-Kod1} \label{thm:fakeCont-kod1-diffbir-descent}
For every pair of coprime integers $1<a<b$, the surface $S(a,b)$ is
a fake real plane of Kodaira dimension
$\kappa=\kappa_{\mathbb{R}}=1$, with contractible complexification.
Furthermore, if $(a,b)\neq (a',b')$ then $S(a,b)$ and $S(a',b')$ are
not birationally diffeomorphic.
\end{proposition}

\begin{proof} The first assertion follows from \cite[Theorem 3.2]{DM15} and Proposition~\ref{Prop:KappaR-properties}\,\hyperref[kapparInvariant]{(1)} using the fact that $B_{\mathbb{R}}(a,b)=B(a,b)$ is a tree. Set $S=S(a,b)$ and $S'=S(a',b')$ and  suppose that there exists a birational diffeomorphism $\varphi:S\dashrightarrow S'$. Let $(V,B)=(V(a,b),B(a,b))$ and $(V',B')=(V(a',b'),B(a',b'))$ be the smooth pairs obtained by taking the minimal resolutions of the pencils $\Psi\circ\beta$ and $\Psi'\circ\beta'$ respectively.
The structures of $B$ and $B'$ imply that $\varphi$ extends to a
birational diffeomorphism of pairs $\Phi:(V,B)\dashrightarrow
(V',B')$. Indeed, otherwise either $\Phi$ or its inverse, say
$\Phi$, would contract an irreducible component of $B$ onto a real
point of $B'$. But $B$ does not contain any irreducible curve whose
proper transform by a birational morphism $W\rightarrow V$ defined
over $\mathbb{R}$ whose center is supported on $B$ is $(-1)$-curve
which can be contracted while keeping the property that the total
transform of $B$ is an SNC divisor.

By virtue of \cite[Lemma 4.5.3 p. 237]{MiyBook}, the positive part
of the Zariski decomposition of $K_V+B$ is equal to
$(1-\frac{1}{a}-\frac{1}{b})\ell$ where $\ell$ denotes a general
real fiber of the $\mathbb{P}^{1}$-fibration
$f:V\rightarrow\mathbb{P}_{\mathbb{R}}^{1}$. Since $1<a<b$, it
follows that $f$ coincides with the log-canonical fibration
$f_{|m(K_{V}+B)|}:V\rightarrow\mathbb{P}_{\mathbb{R}}^{1}$ for every
integer $m\geq 1$. The same holds for the log-canonical fibration
$f'=f_{|m(K_{V'}+B')|}:V'\rightarrow\mathbb{P}_{\mathbb{R}}^{1}$ on
$V'$. Since $B=B_{\mathbb{\mathbb{R}}}$ and similarly for $B'$, it
follows from Lemma \ref{lem:Invariance-kappa} that for every $m\geq
1$, $\Phi$ induces an isomorphism between $H^{0}(V,m(K_{V}+B))$ and
$H^{0}(V',m(K_{V'}+B'))$. Consequently, there exists an automorphism
$\gamma$ of $\mathbb{P}_{\mathbb{R}}^{1}$ defined over $\mathbb{R}$
such that $f'\circ\Phi=\gamma\circ f$.

The curves $E$, $L_{x}$ and $L_{y}$ have multiplicities $1$, $a$ and
$b$ as irreducible components of the scheme theoretic fibers of
$f:V\rightarrow \mathbb{P}^1_{\mathbb{R}}$ over the points $[1:1]$,
$[1:0]$ and $[0:1]$ respectively. Similarly, the curves $E'$,
$L'_{x}$ and $L'_{y}$ have multiplicities $1$, $a'$ and $b'$ as
irreducible components of the scheme theoretic fibers of
$f':V'\rightarrow \mathbb{P}^1_{\mathbb{R}}$ over these points.
Since $1<a<b$ and $1<a'<b'$, and $\Phi(S(\mathbb{R}))\subseteq
S'(\mathbb{R})$, it follows that $\Phi_{*}(E)=E'$,
$\Phi_{*}(L_{x})=L'_{x}$ and $\Phi_{*}(L_{y})=L'_{y}$. Thus
$\gamma=\mathrm{id}$, from which we conclude in turn that $a=a'$ and
$b=b'$. \qed
\end{proof}

\section{Families of algebraic models of Kodaira dimensions  $2$ }\label{Kod2}

Here we construct infinite families of pairwise non birationally
diffeomorphic fake real planes of real Kodaira dimension $2$ with
contractible complexifications. We also give examples of
$\mathbb{Z}$-acyclic complex surfaces of log-general type with two
real forms: one of them has negative logarithmic Kodaira real
dimension and is in fact birationally diffeomorphic to $\A_{\r}^{2}$
whereas the other one has real Kodaira dimension $2$, hence is not
birationally diffeomorphic to $\A_{\r}^{2}$.

\subsection{A criterion for isomorphism}
\begin{lemma} \label{lem:canonical-morph-iso} Let $S$ be a fake real plane of real
Kodaira dimension $2$, with $\mathbb{Q}$-acyclic complexification.
Suppose that there exists a smooth SNC completion $(V,B)$ of $S$
defined over $\mathbb{R}$ for which $B=B_{\mathbb{R}}$. Then the
log-canonical rational map
\[
\varphi:V\dashrightarrow\mathrm{Proj}(\bigoplus_{m\geq0}H^{0}(V,m(K_{V}+B_{\mathbb{R}})))
\]
is a morphism, which restricts to an isomorphism between $S$ and its
image.
\end{lemma}

\begin{proof}
By hypothesis, $S_{\mathbb{C}}$ is a smooth $\mathbb{Q}$-acyclic
surface of Kodaira dimension
$\kappa(S_{\mathbb{C}})=\kappa(S)=\kappa_{\mathbb{R}}(S)=2$. By the
Bogomolov-Miyaoka-Yau inequality (see e.g. \cite[Theorem
6.6.2]{MiyBook}) $S_{\mathbb{C}}$ does not contain any topologically
contractible algebraic curve. Since $S_{\mathbb{C}}$ is affine and
rational, it follows from \cite[Lemma 1.5.1 p. 198]{MiyBook} that
the only curves contracted by a
$K_{V_{\mathbb{C}}}+B_{\mathbb{C}}$-MMP ran from
$(V_{\mathbb{C}},B_{\mathbb{C}})$ are irreducible components of
$B_{\mathbb{C}}$. The assumption that $B=B_{\mathbb{R}}$ implies
that such a MMP is defined over $\mathbb{R}$. Let
$h:(V,B)\rightarrow(W,\Delta)$ be the corresponding birational
morphism, where $\Delta=h_{*}B$. Then $(W,\Delta)$ is an lc pair
defined over $\mathbb{R}$, such that $K_{W}+\Delta$ is semi-ample
\cite[Theorem 4.12.1]{MiyBook}, and $h$ restricts to an isomorphism
$S=V\setminus B\stackrel{\simeq}{\rightarrow}W\setminus\Delta$. We
have $K_{V}+B=h^{*}(K_{W}+\Delta)+E$ where $E$ is an effective
$\mathbb{Q}$-divisor supported on the exceptional locus of $h$, and
$h$ induces an isomorphism
$h^{*}:H^{0}(W,m(K_{W}+\Delta))\stackrel{\sim}{\rightarrow}H^{0}(V,m(K_{V}+B))=H^{0}(V,m(K_{V}+B_{\mathbb{R}}))$
for every $m\geq0$. Finally, again due to the fact that
$S_{\mathbb{C}}$ is affine, rational and does not contain any
topologically contractible algebraic curve, it follows from
\cite[Lemma 1.6.1 p. 200]{MiyBook} that the only curves that could
be contracted by the log-canonical morphism
\[
\psi:W\rightarrow\mathrm{Proj}(\bigoplus_{m\geq0}H^{0}(W,m(K_{W}+\Delta)))\simeq\mathrm{Proj}(\bigoplus_{m\geq0}H^{0}(V,m(K_{V}+B_{\mathbb{R}})))
\]
are irreducible components of $\Delta$. So $\varphi=\psi\circ h$
restricts to an isomorphism between $S$ and its image. \qed
\end{proof}

By combining Lemma \ref{lem:Invariance-kappa} and Lemma
\ref{lem:canonical-morph-iso}, we obtain the following:

\begin{proposition}
\label{prop:real-log-gen-diffbir-iso}Let $S$ and $S'$ be fake real
planes of real Kodaira dimension $2$ with $\mathbb{Q}$-acyclic
complexifications. Assume further that there exist SNC minimal
completions $(V,B)$ and $(V',B')$ of $S$ and $S'$ respectively
defined over $\mathbb{R}$ such that $B=B_{\mathbb{R}}$ and
$B'=B'_{\mathbb{R}}$. Then every birational diffeomorphism
$f:S\dashrightarrow S'$ is an isomorphism.
\end{proposition}

\begin{proof}
Let $F:(V,B)\dashrightarrow(V',B')$ be the birational map of pairs
induced by $f$. The hypothesis implies that the boundaries
$B_{\mathbb{R}}=B$ and  $B'_{\mathbb{R}}=B'$ are trees of
$\mathbb{R}$-rational curves \cite[Lemma 2.3]{DM15}. So by
\ref{lem:Invariance-kappa}, $\varphi$ induces an isomorphism
$\theta$ between the log-canonical rings
\[
R(V,B_{\mathbb{R}})=\bigoplus_{m\geq0}H^{0}(V,m(K_{V}+B_{\mathbb{R}}))\quad\textrm{and\quad}R(V',B_{\mathbb{R}}')=\bigoplus_{m\geq0}H^{0}(V',m(K_{V'}+B_{\mathbb{R}}'))
\]
of the pairs $(V,B_{\mathbb{R}})$ and $(V',B_{\mathbb{R}}')$
respectively. On the other hand, it follows from  Lemma
\ref{lem:canonical-morph-iso} that the log-canonical morphisms
$\varphi:V\rightarrow X=\mathrm{Proj}(R(V,B_{\mathbb{R}}))$ and
$\varphi':V'\rightarrow X'=\mathrm{Proj}(R(V',B_{\mathbb{R}}'))$
restrict to isomorphisms $S=V\setminus B_{\mathbb{R}}\simeq
X\setminus\varphi_{*}(B_{\mathbb{R}})$ and $S'=V'\setminus
B_{\mathbb{R}}'\simeq X'\setminus\varphi'_{*}(B_{\mathbb{R}}')$. We
thus  get a commutative diagram
\[\xymatrix{ S \ar[r] \ar[d]_{f} & V \ar@{-->}[d]_{F} \ar[r]^{\varphi} & X \ar[d]^{\wr} \\ S'  \ar[r] & V' \ar[r]^{\varphi'} & X'}\]
where the right-hand side isomorphism is induced by $\theta$. This
shows that $f$ is an isomorphism. \qed
\end{proof}

\subsection{Miyanishi-Sugie surfaces: a countable family of pairwise non birationally diffeometric fake
real planes of log-general type}

We consider the following real counterpart of a family of smooth
complex topologically contractible surfaces of log-general type
constructed by Miyanishi-Sugie \cite{MS90}. For each integer $s\ge
1$, we will construct a surface $S_s$, defined over $\r$, which
corresponds to the surface $X_{s+1,1}^{(1)}$ of \cite{MS90}, i.e.~to
the construction of \cite{MS90} with $n=s+1$, $m=1$ and $r=s$. We
recall the construction here for self-containedness.

We define $C_s$ and $L_s$ to be the irreducible curves in
$\mathbb{P}_{\mathbb{R}}^{2}=\mathrm{Proj}(\mathbb{R}[x,y,z])$ given
by the zero loci of the polynomials
\[
y^s((s^2-1)x+s y-z)+(x-s y)(x+y)^s\text{ and }((s^2-1)x+s y-z)\]
respectively. Note that the polynomials above correspond to $
-y^{n-1}z+x^{2}(x^{n-2}-\sum_{i=2}^{n-2}(i-1)\binom{n}{i}x^{n-i-2}y^{i})$
and $(n^2-2n)x+(n-1)y-z$, with $n=s+1$, and thus the equations of
$C_s$ and $L_s$ are the same as those given in \cite[Lines 1-2, Page
338]{MS90}.

The curve $C_{s}$ is rational, of degree $s+1$ with a unique
cuspidal singularity of multiplicity $s$ at the point $[0:0:1]$,
which is solved by one blow-up. The line $L_s$ is the tangent line
to $C_{s}$ at the point $p=[1:-1:s^2-s-1]$ and intersects $C_{s}$
with multiplicity $s$ at $p$ and with multiplicity $1$ at the point
$q=[s:1:s^3]$.

Let $\tau:V_{0,s}\rightarrow\mathbb{P}_{\mathbb{R}}^{2}$ be the
birational morphism defined over $\mathbb{R}$ obtained by first
blowing-up $q$ and then blowing-up $s+1$ times the intersection
point of the proper transform of $C_{s}$ with that of the previous
exceptional divisor produced. Let
$E_{1},\ldots,E_{s+2}\simeq\mathbb{P}_{\mathbb{R}}^{1}$ be the
corresponding successive exceptional divisors. Let
$B_{0,s}=C_{s}\cup L_{s}\cup\bigcup_{i=1}^{s+1}E_{i}$, where we
identified each curve with its proper transform in $V_{0,s}$ and let
$S_{s}=V_{0,s}\setminus B_{0,s}$. The dual graph of $B_{0,s}$ is
given in the following figure, where the double arrow corresponds to
a multiple intersection at a point (corresponding here to the point
$p$).

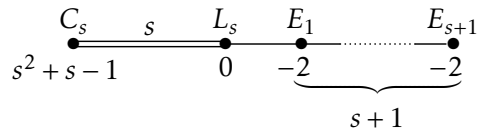
\begin{center}
\xy (-60,5)*{};
 (-10,0)*{};(10,0)*{}**\dir2{-};
 (10,0)*{};(25,0)*{}**\dir{-};
(25,0)*{};(35,0)*{}**\dir{..};(35,0)*{};(40,0)*{\bullet}**\dir{-};
(-10,0)*{\bullet}; (10,0)*{\bullet};(20,0)*{\bullet};
(-11,-3)*{s^2+s-1};(10,-3)*{0};(19,-3)*{-2};(39,-3)*{-2};
(-10,3)*{C_s};(10,3)*{L_s};(20,3)*{E_1};(40,3)*{E_{s+1}};
(0,2)*{\scriptsize{s}};
\POS (30,-7.5)*\txt{$\underbrace{\hspace{2.2cm}}_{\mbox{$s+1$}}$};
\endxy\end{center}
 The smooth surface $S_{s}$
is defined over $\mathbb{R}$. A minimal log-resolution
$f_{s}:(V_{s},B_{s})\rightarrow(V_{0,s},B_{0,s})$ defined over
$\mathbb{R}$ of the pair $(V_{0,s},B_{0,s})$ is obtained by taking a
log-resolution of the singular point $[0:0:1]$ of $C_{s}$ and
blowing-up the real point $p$ and its infinitely near points $s$
times to separate the proper transforms of $C_{s}$ and $L_{s}$.

After blowing-up the singular point of $C_s$ on $V_{0,s}$, the total
transform of $B_{0,s}$ is given by the following dual graph, where
$F$ is the exceptional curve contracted on the singular point of
$C_s$.
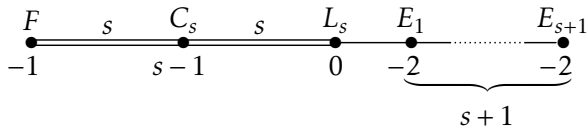
\begin{center}
\xy (-60,5)*{};
 (-30,0)*{};(10,0)*{}**\dir2{-};
 (10,0)*{};(25,0)*{}**\dir{-};
(25,0)*{};(35,0)*{}**\dir{..};(35,0)*{};(40,0)*{\bullet}**\dir{-};
(-30,0)*{\bullet};(-10,0)*{\bullet};
(10,0)*{\bullet};(20,0)*{\bullet};
(-31,-3)*{-1};(-10.5,-3)*{s-1};(10,-3)*{0};(19,-3)*{-2};(39,-3)*{-2};
(-30,3)*{F};(-10,3)*{C_s};(10,3)*{L_s};(20,3)*{E_1};(40,3)*{E_{s+1}};
(0,2)*{\scriptsize{s}};(-20,2)*{\scriptsize{s}};
\POS (30,-7.5)*\txt{$\underbrace{\hspace{2.2cm}}_{\mbox{$s+1$}}$};
\endxy
\end{center}

The total transform $B_{s}$ of $B_{0,s}$ is thus a tree of
$\mathbb{R}$-rational curves whose dual graph is depicted on
Figure~\hyperref[fig:MS-Surfaces]{4} (which is the same as the graph
of $X_{n,m}^{(1)}$ given in \cite[Theorem 1, Page 339]{MS90}, with
$m=1$ and $n=s+1$).

\begin{figure}[ht]
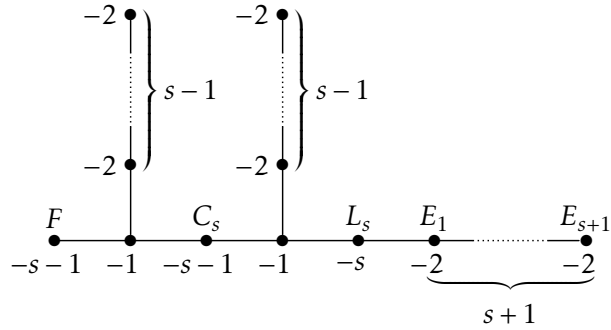

\xy (-85,5)*{};
 (-30,0)*{};(25,0)*{}**\dir{-};
(25,0)*{};(35,0)*{}**\dir{..};(35,0)*{};(40,0)*{\bullet}**\dir{-};
(-20,0)*{\bullet};(-20,15)*{}**\dir{-};(-20,10)*{\bullet};
(0,0)*{\bullet};(0,15)*{}**\dir{-};(0,15)*{};(0,25)*{}**\dir{..};(0,25)*{};(0,30)*{\bullet}**\dir{-};
(-20,15)*{};(-20,25)*{}**\dir{..};(-20,25)*{};(-20,30)*{\bullet}**\dir{-};
(-30,0)*{\bullet};(-10,0)*{\bullet};(10,0)*{\bullet};(20,0)*{\bullet};(0,10)*{\bullet};
(-31,-3)*{-s-1};
(-21,-3)*{-1};(-11,-3)*{-s-1};(-1,-3)*{-1};(9,-3)*{-s};(19,-3)*{-2};(39,-3)*{-2};
(-24,10)*{-2};(-24,30)*{-2};(-4,10)*{-2}; (-4,30)*{-2};
(-30,3)*{F};(-10,3)*{C_s};(10,3)*{L_s};(20,3)*{E_1};(40,3)*{E_{s+1}};
\POS (30,-7.5)*\txt{$\underbrace{\hspace{2.2cm}}_{\mbox{$s+1$}}$};
\POS (2.5,30.5).(2.5,9.5)!C*\frm{\}};(8,20)*{s-1}; \POS
(-17.5,30.5).(-17.5,9.5)!C*\frm{\}};(-12,20)*{s-1};
\endxy
\caption{The weighted dual graph of the boundary divisor $B_s$.}
\vspace{-5.5cm}\phantomsection\label{fig:MS-Surfaces}\vspace{5.5cm}
\end{figure}

\begin{proposition}
For every $s\ge 2$, $S_{s}$ is a fake real plane with contractible
complexification and $\kappa(S_{s})=\kappa_{\mathbb{R}}(S_{s})=2$.
Furthermore, if $s,s'\ge 2$ are two integers, then $S_{s}$ is
birationally diffeomorphic to $S_{s'}$ if and only if $s=s'$.
\end{proposition}

\begin{proof}
The fact that the complexification of $S_{s}$ is contractible and of
log-general type is proven in \cite{MS90}: the log-general type is
the purpose of the construction and the contrability is given in
\cite[Theorem 2]{MS90}. Since $B_{s}=B_{s,\mathbb{R}}$ by
construction, we have $\kappa_{\mathbb{R}}(S_{s})=\kappa(S_{s})=2$,
and $S_{n}$ is a fake real plane by virtue of \cite[Proposition
2.4]{DM15}. By Proposition \ref{prop:real-log-gen-diffbir-iso},
every birational diffeomorphism between $S_{s}$ and $S_{s'}$ is an
isomorphism. But the description of the dual graphs of the
boundaries $B_{s}$ in Figure \hyperref[fig:MS-Surfaces]{4} implies
that every isomorphism between $S_{s}$ and $S_{s'}$ extends to an
isomorphism of pairs between $(V_{s},B_{s})$ and $(V_{s'},B_{s'})$
and that two such pairs are non isomorphic for different $s$ and
$s'$. \qed
\end{proof}

\subsection{Fake real planes of general type with nontrivial real forms}\label{SubSecNontrivialforms}
To complete this section, we reconsider a family of fake real planes
of general type with two real forms intoduced in \cite[\S
3.2.2]{DM15}. We start with the projective duals $\Gamma_{1}$ and
$\Gamma_{2}$ of real nodal cubic curves
$C_{1},C_{2}\subseteq\mathbb{P}_{\r}^{2}$, such that the two
branches at the singular point of $C_1,C_2$ are real, respectively
non-real.

Note that $C_1,C_2$ are not equivalent under
$\Aut(\p^2_\r)=\mathrm{PGL}_2(\r)$, and that every real nodal cubic
curve of $\p^2_\r$ is projectively equivalent to either $C_1$ or $C_2$.
The latter can be checked by looking at the parametrisations
$\p^1_\r\to C_1,C_2$, given by polynomials of degree $3$ having the
same value at two points, which are either real or pairs of non-real
complex conjugates. Explicitely, one can choose, for instance, the
equations of $C_1$ and $C_2$ to be
\begin{center}
$x^2z-y^2z+xy^2=0 \mbox{ and }x^2z+y^2z-xy^2=0$.\,\footnote{The
equations given in \cite[\S 3.2.2]{DM15} are unfortunately false, as
both are linearly equivalent, having one singular point with two
real branches.}\end{center} With these coordinates, we find that
$C_1$ and $C_2$ are exchanged by the non-real complex projective
transformation $[x:y:z]\mapsto [x:\mathbf{i}y:z]$. Moreover, both
curves $C_1,C_2$ have a singular point at $[0:0:1]$, an inflection
point at $[0:1:0]$ and then two other complex inflection points,
which are $[1:\pm \mathbf{i}\frac{\sqrt{3}}{3}:\frac{1}{4}]$ for
$C_1$ and $[1:\pm \frac{\sqrt{3}}{3}:\frac{1}{4}]$ for $C_2$.

The curves $\Gamma_{1}$ and $\Gamma_{2}$ are thus rational quartics
with three cusps: an ordinary real cusp $p_{0}$ corresponding to the
common $\r$-rational flex of $C_{1}$ and $C_{2}$, and either a pair
of non-real conjugate cusps $q$ and $\overline{q}$ for $\Gamma_{1}$
or an additional pair or real ordinary cusps $q_{1}$ and $q_{2}$ for
$\Gamma_{2}$. So $\Gamma_{1}$ and $\Gamma_{2}$ are not isomorphic
over $\r$, but their respective complexifications are both
projectively equivalent over $\C$. In fact, after change of
coordinates, the curves $\Gamma_1$ and $\Gamma_2$ can be given by
the equations
\[(x^2+y^2)^2+z(2x^3+2xy^2-y^2z)=0\text{ and }x^2y^2+x^2z^2+y^2z^2+2xyz(x+y-z)=0 ,\]
and the projective transformation $\theta\colon [x:y:z]\mapsto
[x+\mathbf{i}y:x-\mathbf{i}y:z/2]$ maps $\Gamma_1$ isomorphically
onto $\Gamma_2$. The cusps of $\Gamma_1$ are then $p_0=[0:0:1],
q=[1:\mathbf{i}:0]$, $\overline{q}=[1:-\mathbf{i}:0]$, and the ones
of $\Gamma_2$ are $p_0=[0:0:1]$, $q_1=[1:0:0]$, $q_2=[0:1:0]$.

For $i=1,2$, the tangent line $L_i=T_{p_{0}}(\Gamma_{i})$ to
$\Gamma_{i}$ at $p_{0}$ (given respectively by $y=0$ and $x=y$ and
satisfying $\theta(L_1)=L_2$) intersects $\Gamma_{i}$ transversally
in a unique other real point $p_i$ different from $p_{0}$ (being
given by $p_1=[1:0:-1/2]$ and $p_2=[1:1:-1/4]=\theta(p_1)$). Let
$(a,b)$ be a pair of positive integers such that $4b-a=\pm1$ and let
$\tau_{i}:V_{i}\rightarrow\mathbb{P}_{\r}^{2}$ be the real
birational morphism obtained by first blowing-up $p_i$ with
exceptional divisor $E_{1}\simeq \mathbb{P}^1_{\r}$ and then
blowing-up a sequence of real points on the successive total
transforms of $E_{1}$ in such a way that the following two
conditions are satisfied: a) the inverse image of $p_i$ is a chain
of curves isomorphic to $\mathbb{P}_{\r}^{1}$ containing a unique
$(-1)$-curve $A$ and b) the coefficients of $A$ in the total
transform of $\Gamma_{i}$ and $L_i=T_{p_{0}}(\Gamma_{i})$ are equal
to $a$ and $b$ respectively. We denote the corresponding exceptional
divisors by $E_{1},\ldots,E_{r-1},E_{r}=A$ and we let
$B_{i}=\Gamma_{i}\cup
T_{p_{0}}(\Gamma_{i})\cup\bigcup_{j=1}^{r-1}E_{j}$, $i=1,2$.
\begin{proposition}
For every choice of integers $(a,b)\in\mathbb{Z}_{>0}$ as above, the
following hold for the surfaces $S_{i}=V_{i}\setminus B_{i}$,
$i=1,2$:
\begin{enumerate}
\item[\rm a)] $S_{1}$ and $S_{2}$ are $\mathbb{Z}$-acyclic fake real planes of
Kodaira dimension $2$ with isomorphic complexifications.

\item[\rm b)] $\kappa_{\r}(S_{2})=2$, in particular $S_{2}$ is not
birationally diffeomorphic to $\A_{\r}^{2}$.

\item[\rm c)] $\kappa_{\r}(S_{1})=-\infty$, and $S_{1}$ is actually
birationally diffeomorphic to $\A_{\r}^{2}$.
\end{enumerate}
\end{proposition}

\begin{proof}
The complex surfaces $S_{1,\C}$ and $S_{2,\C}$ are isomorphic, by
lifting the projective transformation $\theta$. The fact that
$S_{1}$ and $S_{2}$ are $\mathbb{Z}$-acyclic fake real planes of
log-general type is established in \cite[Proposition 3.10]{DM15}.
Since by construction $B_{2}$ is a tree of $\mathbb{R}$-rational
curves, we have $\kappa_{\r}(S_{2})=\kappa(S_{2})=2$ by
Proposition~\ref{Prop:KappaR-properties}\,\hyperref[kapparsmallerkappa]{(2)},
and so $S_{2}$ is not birationally diffeomorphic to $\A_{\r}^{2}$.
The fact that $S_{1}$ is birationally diffeomorphic to $\A_{\r}^{2}$
is proven in \cite[Proposition 21]{DMSpitz}. \qed
\end{proof}

\providecommand{\bysame}{\leavevmode\hbox
to3em{\hrulefill}\thinspace}
%
%

\bibliographystyle{amsalpha}
\bibliographymark{References}
\def\cprime{$'$}

\end{document}